\RequirePackage[final]{graphicx}
 \documentclass[oneside, usenames, dvipsnames, svgnames, table, final]{amsart}

\usepackage[T1]{fontenc}
\usepackage[utf8]{inputenc}

\usepackage{fourier}
\usepackage{longtable}
\usepackage{overpic}

\usepackage{etoolbox}
\usepackage{ifdraft}

\usepackage[colorlinks=true, citecolor=green!50!black, linkcolor=red!50!black, pagebackref=false, final]{hyperref}

\ifdraft{\usepackage[color,notref, notcite]{showkeys}}{}
\usepackage[letterpaper]{geometry}
\usepackage{xargs}
\usepackage{amsmath, amsthm, mathrsfs, amssymb}
\usepackage{xypic}
\usepackage{tikz-cd} 
\usepackage[final]{graphicx}
\usepackage{enumitem}
\usepackage{bm}
\usepackage{xcolor}
\usepackage[en-AU]{datetime2}
\usepackage[T1]{fontenc}

\newcommand{\mathbx}[1]{\mathbb{#1}}

\usepackage[colorinlistoftodos,prependcaption,textsize=tiny]{todonotes} 

\def\showkeyslabelformat[#1]{%
   {\tiny #1}\par}

\theoremstyle{plain}
\ifcscounter{chapter}{%
        \newtheorem{thm}{Theorem}[chapter]%
        }%
        {%
        \newtheorem{thm}{Theorem}[]%
        }
\newtheorem{theorem}[thm]{Theorem}

\newtheorem{lemma}[thm]{Lemma}

\newtheorem{corollary}[thm]{Corollary}

\newtheorem{proposition}[thm]{Proposition}

\newtheorem*{claim*}{Claim} 
\newtheorem*{thm*}{Theorem}

\theoremstyle{definition}

\newtheorem{notation}[thm]{Notation}

\theoremstyle{remark}

\newtheorem{remark}[thm]{Remark}

\newtheorem{conjecture}[thm]{Conjecture}

\newcommand{\R}{\mathbx{R}}
\newcommand{\Z}{\mathbx{Z}}

\newcommand{\RR}{\R}
\newcommand{\ZZ}{\Z}

\newcommand{\HH}{\mathbx{H}}

\newcommand{\OG}{\operatorname{O}}
\newcommand{\Og}{\OG}

\newcommand{\SO}{\operatorname{SO}}

\newcommand{\sm}{\setminus}

\newcommand{\im}{\operatorname{im}}

\renewcommand{\vec}[1]{\mathbf{#1}}

\usepackage{colortbl}
\usepackage{afterpage}
\usepackage{pdflscape}
\usepackage{changes}

\definecolor{highlight}{rgb}{0.99,0.96,0.94}

\usepackage[backref=true,backend=biber,style=alphabetic,citestyle=alphabetic,url=false]{biblatex}
\newcommand{\lalog}[2][]{\ifdraft{\todo[linecolor=Red,backgroundcolor=Red!25,bordercolor=Red,#1]{#2---Lalo G.}}{}}
\newcommand{\benw}[2][]{\ifdraft{\todo[linecolor=Green,backgroundcolor=Green!25,bordercolor=Green,#1]{#2---Ben W.}}{}}

\newcommand{\w}{\omega}

\newcommand{\Homeo}{\operatorname{Homeo}}
\newcommand{\Diff}{\operatorname{Diff}}

\newcommand{\link}{\operatorname{lk}}

\author{Luis Eduardo Garc\'ia-Hern\'andez}
\address{Instituto de Matemáticas UNAM\\
Ciudad Universitaria
Coyoacán, CDMX 04510, MÉXICO.
}
\email[L. E. Garc\'ia-Hern\'andez]{legh@ciencias.unam.mx}

\author{Ben Williams}
\address{Department of Mathematics\\
  1984 Mathematics Rd \\
  Vancouver BC V6T 1Z2\\
CANADA.}
\email[B.~Williams]{tbjw@math.ubc.ca}

\begin{document}   


\thanks{We acknowledge the support of the Natural Sciences and Engineering Research Council of Canada (NSERC), RGPIN-2021-02603.}

\title{Linear and smooth oriented equivalence of orthogonal representations of finite groups}

\begin{abstract}
  Let $\Gamma$ be a finite group. We prove that if $\rho , \rho': \Gamma \to \Og(4)$ are two representations
  that are conjugate by an orientation-preserving diffeomorphism of $S^3$, then they are conjugate by an element of
  $\SO(4)$. In the process, we prove that if $G \subset \Og(4)$ is a finite group, then exactly one of the following is
  true: the elements of $G$ have a common invariant $1$-dimensional subspace in $\R^4$; some element of $G$ has no invariant
  $1$-dimensional subspace; or $G$ is conjugate to a specific group $K \subset \Og(4)$ of order $16$.
\end{abstract}

\maketitle

\section{Introduction}
Let $n$ be a positive integer. The orthogonal group $\Og(n)$ is the group of isometries of $S^{n-1}$, and $\SO(n)$ is the group of orientation-preserving isometries. The group $\Og(n)$ is a subgroup of $\Diff(S^{n-1})$, the group of diffeomorphisms of $S^{n-1}$, and $\SO(n) = \Og(n) \cap \Diff^{+}(S^{n-1})$, where $\Diff^{+}$ indicates the group of orientation-preserving diffeomorphisms.

Let $\Gamma$ be a (discrete) finite group, and let $G, H$ be subgroups of $\Diff(S^{n-1})$. We say two homomorphisms $\rho, \rho' : \Gamma \to G$ are \emph{$H$-equivalent} if there exists some $h \in H$ such that $h^{-1} \rho(\gamma) h = \rho'(\gamma)$ for all $\gamma \in \Gamma$.

We make the following conjecture, a special case of which arose in the process of classifying types of symmetries of oriented knots in
$S^3$ in \cite{Boyle2023}.
\begin{conjecture} \label{conj:mainConj}
     If $\rho, \rho' : \Gamma \to \Og(n)$ are $\Og(n)$-equivalent and $\Diff^{+}(S^{n-1})$-equivalent, then they are $\SO(n)$-equivalent.
   \end{conjecture}
If $\rho : \Gamma \to \Og(n)$, then it will be convenient to say this conjecture holds ``for $\rho$'' if it holds for
all pairs $(\rho, \rho')$.

When $n \le 5$, \cite{Cappell1999} proves that  $\rho, \rho' : \Gamma \to \Og(n)$ are $\Og(n)$-equivalent if they are
$\Diff(S^{n-1})$-equivalent (in fact, if they are $\Homeo(S^{n-1})$-equivalent). This is part of a larger literature on
determining when topological and linear equivalence of operators or representations of finite groups coincide. See for
instance \cite{Kuiper1973}, \cite{Cappell1981}, \cite{Hambleton2005}.

When $n\le 5$, therefore, our conjecture
can be restated:
\begin{quote}
   If $\rho, \rho' : \Gamma \to \Og(n)$ are $\Diff^{+}(S^{n-1})$-equivalent, then they are $\SO(n)$-equivalent.
\end{quote}

\medskip

Linear equivalence, i.e., $\Og(n)$-equivalence, is thoroughly studied. There are some homomorphisms $\rho : \Gamma \to \Og(n)$ for which the $\Og(n)$- and $\SO(n)$-equivalence classes coincide
. For these homomorphisms, Conjecture \ref{conj:mainConj} holds trivially. Otherwise, the $\Og(n)$-equivalence class of $\rho$ splits into two $\SO(n)$-equivalence classes.

The simplest example of a homomorphism whose $\Og(n)$-equivalence class decomposes as two $\SO(n)$-equivalence classes is that of 
\[ \rho_i : C_m = \langle t \mid t^m = 1 \rangle \to \Og(2), \quad \rho(t) = \begin{bmatrix}
    \cos \frac{2 i\pi}{m} & - \sin \frac{2 i\pi}{m} \\ 
    \sin \frac{2 i\pi}{m} & \cos \frac{2 i\pi}{m} 
\end{bmatrix} \]
when $m \ge 3$ and $i= \pm 1$. Conjugation by $w=\begin{bmatrix} -1 & 0 \\ 0 & 1 \end{bmatrix}$ takes $\rho_1$ to $\rho_{-1}$, so that these homomorphisms are $\Og(2)$-equivalent. They take values in $\SO(2)$, which is abelian, so that we immediately conclude that they are not $\SO(2)$-equivalent. Conjecture \ref{conj:mainConj} says that these homomorphisms are not $\Diff^{+}(S^1)$-equivalent.

Indeed, we can find a topological invariant distinguishing $\rho_{1}$ from $\rho_{-1}$. The two homomorphisms give actions of $C_m$ on $S^1$. Mark a point $p \in S^1$. We may list the points in the orbit of $p$ under the $C_m$-action, starting with $p$ and proceeding around the oriented $S^1$ in the positive direction. In the case of $\rho_{1}$, the list is $(p, t \cdot p, t^2 \cdot p, \dots )$ and in the case of $\rho_{-1}$, the list is $(p, t^{m-1}\cdot p, t^{m-2}\cdot p, \dots )$. These lists, and in particular, which of $t\cdot p$ or $t^{m-1}\cdot p$ is first encountered on a path starting at $p$ and proceeding in the positive direction, are invariants of the $\Diff^+(S^1)$-equivalence class of the representation, so that $\rho_{1}$ and $\rho_{-1}$ are not equivalent in this sense. An illustrative graphic in the case of $m=5$ is given in Figure \ref{fig:C5S1}.
\begin{figure}
    \centering
    \begin{minipage}{0.4\textwidth}
  \centering
  \includegraphics[width=\linewidth]{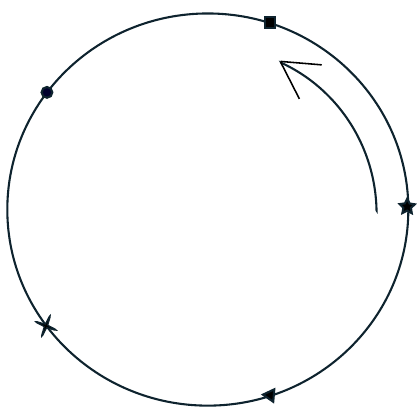}
  \caption{figure}{An oriented circle, a marked point $p$ (indicated by $\star$) and its orbit under a $C_5$-action by isometries. The point $\blacksquare$ is $t \cdot p$.}
  \label{fig:C5S1}
\end{minipage} \hfil
\begin{minipage}{.4\textwidth}
  \centering
  \includegraphics[width=\linewidth]{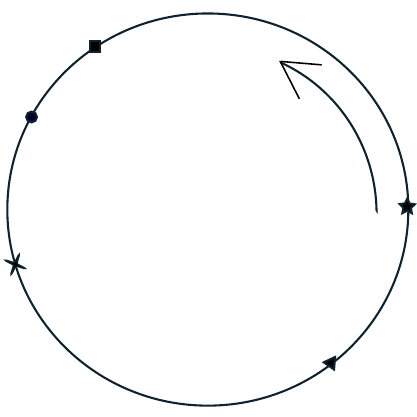}
  \caption{figure}{An oriented circle, a marked point $p$ (indicated by $\star$) and its orbit under a $C_5$-action $\Diff^+(S^1)$-equivalent to the action of Figure \ref{fig:C5S1}. The point $\blacksquare$ is $t \cdot p$.}
  \label{fig:C5S1v}
\end{minipage}
\end{figure}

This paper proves Conjecture \ref{conj:mainConj} when $n \le 5$. If $n$ is odd, then the conjecture happens to be trivially true. Case $n=2$ is settled by the argument sketched above for $C_m$, which leaves only $n=4$ requiring substantial work.

\begin{remark}
  Besides being intrinsically interesting, the $n=4$ case of Conjecture \ref{conj:mainConj} is likely to be useful in classifying the
  types of symmetries that a link in $S^3$ may possess, generalizing what was done in \cite{Boyle2023} for knots.
\end{remark}
\subsection*{Centralizers}

For the purposes of establishing Conjecture \ref{conj:mainConj}, there is nothing to be lost in assuming $\rho$ and
$\rho'$ are injective homomorphisms, since $\Gamma$ can be replaced by $\Gamma/(\ker \rho)$. From now on, we consider
only injective homomorphisms $\rho : \Gamma \to \Og(n)$. It is possible to restate Conjecture \ref{conj:mainConj} in terms of subgroups of $\Og(n)$ only. 

\begin{notation}
    We say that a subgroup $G \subset \Og(n)$ has the \emph{amphichiral linear centralizer property (has LCP)} if there
    exists some $w \in \Og(n) \sm \SO(n)$ such that $w^{-1}g w = g$ for all $g \in G$, i.e., if the centralizer of $G$
    in $\Og(n)$ is not contained in $\SO(n)$. Equivalently, $G$ has LCP if the inclusion homomorphism $G \to \Og(n)$ is linearly amphichiral.

    Similarly, we say that a subgroup $G \subset \Diff(S^{n-1})$ has the \emph{amphichiral smooth centralizer property
      (has SCP)} if there exists some $w \in \Diff(S^{n-1})\sm \Diff^{+}(S^{n-1})$ such that $w^{-1}g w = g$ for all $g \in G$.
\end{notation}
If $H$ is a subgroup of a group having LCP (resp.~SCP), then $H$ also has LCP (resp.~SCP). 
\begin{conjecture} \label{conj:subgroup}
    Let $G$ be a finite subgroup of $\Og(n)$. If $G$ has SCP, then it has LCP. Equivalently, if the centralizer of $G$ in $\Diff(S^{n-1})$ is not contained in $\Diff^+(S^{n-1})$, then it contains an element in $\Og(n) \sm SO(n)$.
\end{conjecture}

\begin{proposition} \label{pr:equivalence}
    Conjecture \ref{conj:mainConj} holds for $\rho : \Gamma \to \Og(n)$ if and only if Conjecture \ref{conj:subgroup} holds for $\im(\rho)$
\end{proposition}
The proof is elementary.\benw{A version is commented out.}\lalog{Should we add the proof?}


Conjecture \ref{conj:subgroup} depends only on the $\Og(n)$-conjugacy class of $G$. Furthermore, Conjecture
\ref{conj:subgroup} holds for all $G$ if $n$ is odd, since $-I_n \in Z(\Og(n))$ in this case. We concentrate on the case
of even $n$ from now on.

The following is proved in \cite{Boyle2023}.\lalog{I think it's Proposition 7.4}\benw{Since the reference is still under review, we shouldn't be specific.}
\begin{proposition} \label{pr:chiralRepClassifier}
  Suppose $G \subset \Og(n)$ is a finite subgroup. The following are equivalent:
  \begin{enumerate}
  \item \label{pia1} $G$ has LCP;
  \item \label{pia4} The action of $G$ on $\R^n$ admits an odd-dimensional invariant subspace.
  \end{enumerate}
\end{proposition}

\section{Cyclic groups in low dimension}
\label{sec:cyclicLow}

Since all irreducible $\R$-representations of cyclic groups are either $1$- or $2$-dimensional, Proposition
\ref{pr:chiralRepClassifier} shows that a finite cyclic subgroup $\langle g \rangle \subset \Og(n)$ has LCP if and only
if a $g$ has a real eigenspace.

The following proposition furnishes a method for proving Conjecture \ref{conj:subgroup} for cyclic groups.
\begin{proposition} \label{pr:methodProp}
    Let $n$ be an even integer. Let $A$ be the subset of $\Og(n)$ that consists of all finite-order elements $g \in \Og(n)$
    that do not have any real eigenspace. Suppose there exists a function 
    \[ i : A \to B \]
    where $B$ is some set, such that:
    \begin{enumerate}
    \item\label{i:mp1} if there exists $d \in \Diff^+(S^3)$ so that $d^{-1} g d \in A$, then $i(d^{-1} g d) = i(g)$;
    \item \label{i:mp2} $i$ is not invariant under conjugation by some element $l \in \Og(n) \sm \SO(n)$.
    \end{enumerate}
    Then Conjecture \ref{conj:subgroup} holds for all finite cyclic subgroups of $\Og(n)$.
\end{proposition}
\begin{proof}
Suppose $G=\langle g \rangle \subset \Og(n)$ is a finite cyclic subgroup. If this group has LCP, then there is nothing
to do.

Suppose it does not have LCP, so that $g \in A$. The conditions on $i$ imply that $i(w^{-1}gw) \neq i(g)$ whenever $w \in \Diff(S^{n-1})
\sm \Diff^{+}(S^{n-1})$ is such that $w^{-1} g w \in A$, since any such $w$ lies in $l \Diff^{+}(S^{n-1})$. Therefore,
$w^{-1}g w \neq g$ for all $w \in \Diff(S^{n-1}) \sm \Diff^+(S^{n-1})$.
\end{proof}

\begin{proposition} \label{pr:Cyclicn2}
    Conjecture \ref{conj:subgroup} holds for cyclic groups when $n=2$. 
\end{proposition}
\begin{proof}
    We construct an invariant $i$ as in Proposition \ref{pr:methodProp}. If $f \in \Og(2)$ is of finite order but has no
    real eigenspace, then $f$ is a rotation of order $m \ge 3$.

    The construction is a generalization of one in the introduction. Mark a point $p \in S^1$. The orbit
    \[ \{p, f(p), f^2(p), \dots, f^{m-1}(p)\}\] can be ordered by starting at $p$ and proceeding around $S^1$ in the
    positive direction. We obtain an ordered list $$(p, f^{a_1}(p) , f^{a_2}(p), \dots, f^{a_{m-1}}(p))$$ of points on
    the circle, where $(a_1, \dots, a_m)$ is a permutation of $1, \dots, m-1$. It is elementary to verify that this
    permutation of $(1,\dots, m-1)$ is invariant under $\Diff^+(S^1)$-conjugation, but is reversed by conjugation by an
    element of $\Og(2) \sm \SO(2)$.  We therefore have an invariant function $i$ taking values in $B=S_m$, the symmetric
    group on $m$ letters satisfying the conditions of Proposition \ref{pr:Cyclicn2}. This completes the proof.
\end{proof}

We derive the following corollary.
\begin{theorem} \label{th:con2}
  Conjecture \ref{conj:subgroup} holds when $n=2$. 
\end{theorem}
\begin{proof}
  If $\Gamma \subset \Og(2)$ is a finite group, then $\Gamma$ is either cyclic, acting on $S^1$ by rotations, or
  dihedral, acting by rotations and at least one reflection. We established the cyclic case in Proposition \ref{pr:Cyclicn2}.

  If $\Gamma$ is a dihedral group of order $4$, then it is generated by a rotation by $\pi$ and a reflection, $r$. The axis
  of reflection of $r$ furnishes a $1$-dimensional invariant subspace, so $\Gamma$ has LCP.

  If $\Gamma$ is a dihedral group of order $2m \ge 6$, then it contains a cyclic group of order $m \ge 3$. This cyclic
  group $\Gamma'$ is generated by a rotation by $2\pi/m$, which does not have an eigenspace, and therefore $\Gamma'$
  does not have SCP, by \ref{pr:Cyclicn2}, and neither does $\Gamma$.
\end{proof}

\begin{proposition}\label{pr:Cyclicn4}
     Conjecture \ref{conj:subgroup} holds for cyclic groups when $n=4$.
\end{proposition}
This is proved in \cite[Proposition 7.8]{Boyle2023}, by an argument that follows the pattern of Proposition
\ref{pr:methodProp}. Since that work has not yet been published, we sketch this argument here. The set $A$ consists of
the finite-order elements of $\Og(n)$ without real eigenspaces. These are double rotations. For a given $g \in A$ of
order $m$, the space $\R^4$ may be decomposed as an orthogonal direct sum $V_1 \perp V_2$ of $2$-planes such that $g$
acts on $V_1$ by rotation by $a_1/m$-turn and on $V_2$ by rotation by $a_2/m$-turn, where
$a_1,a_2 \in \{0, \dots, m-1\}$, Since $g$ has no real eigenspaces, $a_1,a_2 \neq 0$ and $a_1,a_2 \neq m/2$.

We distinguish two cases: the ordinary case, where $a_1 \not \equiv \pm a_2 \pmod m$, and the isoclinic case, where $a_1\equiv \pm a_2 \pmod m$. 

We construct an invariant function $i$ in the ordinary case first. For an ordinary double-rotation $g$, the decomposition of $\R^4$ into $V_1
\perp V_2$ is unique. Define $2$
$g$-invariant round circles $S_1$ and $S_2$ in $S^3$ on which $g$ acts by an $a_1/m$ -turn and by a $a_2/m$-turn respectively: choose an orthogonal
$V_1 \perp V_2 = \R^4$ as above, and take $S_1$ and $S_2$ to be the unit circles in $V_1$ and $V_2$ respectively. Orient
$S_1$ and $S_2$ by means of the $g$-action: the positive direction on $S_1$ (resp.~$S_2$) is the direction from a marked
point $p$ to $g(p)$ along $S_1$ (resp.~$S_2$) that passes through as few of the points in the $g$-orbit of $p$ as
possible. With these orientations, $\link(S_1, S_2)$ is either $+1$ or $-1$. We claim that the class of
$\link(S_1, S_2) \in \Z/(m)$ is an invariant meeting the hypotheses of Proposition \ref{pr:methodProp}. Certainly, it is
changed by a reflection of $S^3$, which changes the orientation of the ambient space, so Hypothesis \ref{i:mp2} is satisfied.

To see that Hypothesis \ref{i:mp1} of Proposition \ref{pr:methodProp} holds, argue as follows: if $g$
is conjugate to $g' = d^{-1} g d \in A$ where $d \in \Diff^+(S^3)$, then we consider the corresponding round circles
$S'_1$ and $S'_2$ for $g'$. Then $dS'_1$ is a $g$-invariant smooth unknotted circle on which $g$ acts as an
$a/m$-turn. The main theorem of \cite{Freedman1995} says that $dS_1'$ is isotopic through $g$-invariant circles to a
round circle, which must be $S_1$ in this case, since $g$ acts by $a/m$-turn on each of the circles in the
isotopy. Similarly, $dS'_2$ is isotopic through $g$-invariant circles to $S_2$. An elementary homology argument shows
that the isotopies through $g$-invariant circles change the linking numbers by multiples of $m$ only. Therefore
\[ \link(S_1, S_2) \equiv \link(dS_1', dS_2') \equiv \link(S_1',S_2') \pmod m. \]

\smallskip

We now extend the definition of $i$ to all of $A$, including the isotropic rotations. For an isotropic rotation $g$,
there are infinitely many orthogonal decompositions $\R^4 = V_1 \perp V_2$ into $g$-invariant $2$-planes. It is possible
to define disjoint $g$-invariant round circles $S_1$ and $S_2$ as before, but they are no longer unique. Nonetheless,
for any choices $(S_1,S_2)$ and $(S_1'', S_2'')$ of such circles, one can take $S_1$ to $S_1''$ and $S_2$ to $S_2''$ by smooth
$g$-invariant isotopies. The action of $g$ endows any choice of $S_1$ and $S_2$ with orientations, so that the linking
number $\link(S_1, S_2)$ is defined, and consideration of isotopies shows it is independent of the precise choice of
$(S_1, S_2)$. We again claim that the class of $\link(S_1, S_2)$ in $\ZZ/(m)$ satisfies the hypotheses of
\ref{pr:methodProp}. The same argument as before shows that Hypothesis \ref{i:mp2} is satisfied.

To see that Hypothesis \ref{i:mp1} of Proposition \ref{pr:methodProp} holds, argue as follows: if $g$
is conjugate to $g' = d^{-1} g d \in A$ where $d \in \Diff^+(S^3)$, then we consider the corresponding round circles
$S'_1$ and $S'_2$ for $g'$. Then $dS'_1$ is a $g$-invariant smooth unknotted circle on which $g$ acts as an
$a_1/m$-turn. The main theorem of \cite{Freedman1995} says that $dS_1'$ is isotopic through $g$-invariant circles to a
round circle, and that round circle is in turn isotopic through $g$-invariant circles to $S_1$. Similarly, $dS'_2$ is
isotopic through $g$-invariant circles to $S_2$. The same elementary homology argument as before
that the isotopies through $g$-invariant circles change the linking numbers by multiples of $m$ only. Therefore
\[ \link(S_1, S_2) \equiv \link(dS_1', dS_2') \equiv \link(S_1',S_2') \pmod m. \]
In all cases, the class of $\link(S_1, S_2)$ in $\ZZ/(m)$ furnishes the required invariant $i$. 

\begin{corollary} \label{cor:redToCyclic4}
   If $G < \Og(4)$ is a finite group and $g \in G$ does not have a $1$-dimensional invariant subspace, then $G$ does not
   have SCP.
 \end{corollary}
 \begin{proof}
   If $g$ does not have a $1$-dimensional invariant subspace then it does not have an odd-dimensional invariant
   subspace, and so $\langle g \rangle$ does not have LCP, by Proposition \ref{pr:chiralRepClassifier}. Proposition \ref{pr:Cyclicn4} says
   $\langle g \rangle$ does not have SCP. We conclude that the larger group $G$ also does not have SCP.
 \end{proof}




\section{Quaternion geometry}
\label{sec:geometry}

We adopt the notation of \cite{Rastanawi2022}, which is similar to that of \cite{Conway2003}. Identify $\R^4 = \HH$. If $z,w \in \HH$ are quaternions of unit length, then $[z,w] \in \SO(4)$ is the transformation $ x \mapsto \bar z x w$, and $\ast[z.w] \in \Og(4)$ is the transformation $ x \mapsto \bar z \bar x w$, which is a reflection.
It is easy to verify that $[z_1, w_1] = [z_2, w_2]$ if and only if $(z_1, w_1) = \pm (z_2, w_2)$.

Composition of pairs $[z_1, w_1] [z_2, w_2]$ is carried out left to right, which ensures that $([z_1, w_1] [z_2, w_2])x  = [z_1z_2, w_1w_2]x$.

It is useful to know when a rotation $[z,w]$ has no invariant
$1$-dimensional subspace, or equivalently, when $[z,w]$ has no real eigenspace.

\begin{proposition} \label{pr:quaternionCondition}
    The element $[z,y] \in \SO(4)$ has no invariant $1$-dimensional subspace if and only if $z$ is not conjugate in $\HH^\times$ to either $y$ or $-y$.
\end{proposition}
\begin{proof}
    The transformation $[z,y]$ has an invariant $1$-dimensional subspace if and only if the equation $[z,y] (x)= \pm x $
    holds for some $x \in \HH^\times$. This equation is equivalent to
    $z = \pm x^{-1} y x$.
\end{proof}

\begin{remark} \label{rem:sign}
    If $z$ and $y$ are conjugate quaternions, then 
    \[ \Re(y) = \frac{1}{2} (y + \bar y) = \frac{1}{2}(z+\bar z) = \Re(z).\]
\end{remark}

\begin{notation}
    Beyond $1$, $i$, $j$ and $k$, certain unit quaternions have conventional names:
    \begin{align*}
        \w &= \frac{1}{2}(-1 + i + j + k) \quad \text{(of order $3$);} \\
        i_O &= \frac{1}{\sqrt{2}}(j+k) \quad \text{(of order $4$);} \\
        i_I &= \frac{1}{2}\Big(i +\frac{\sqrt{5}-1}{2} j + \frac{\sqrt{5}+1}{2} k\Big)\quad \text{(of order $4$).}
    \end{align*}
\end{notation}

\section{Analysis of cases}
\label{sec:cases}

Our approach to Conjecture \ref{conj:subgroup} for $n=4$ is to analyze cases. Almost all the groups we consider fall into two easy-to-handle cases. There is one exception.

\begin{notation}
    Let $K \subset \Og(4)$ denote the group generated by $k_1=*[i,i][i,1]$ and $k_2=*[k,k][i,1]$.
\end{notation}

This group is denoted $\boxplus^{\mathbf{p2gg}}_{2,2}$ by \cite{Rastanawi2022}. It is a group of order $16$, abstractly
isomorphic to the nontrivial semidirect product $C_4 \rtimes C_4$. We will establish in Proposition \ref{pr:KsmChiral}
that $K$ does not have SCP.

\begin{theorem} \label{th:subgroup}
    Let $G$ be a finite subgroup of $\Og(4)$. Exactly one of the following is true:
    \begin{itemize}
        \item There is a $1$-dimensional $G$-invariant subspace of $\R^4$.
        \item There is some element $g \in G$ that has no $1$-dimensional invariant subspace.
        \item $G$ is conjugate in $\Og(4)$ to the group $K$.
    \end{itemize}
\end{theorem}

In the first case, $G$ has LCP by Proposition \ref{pr:chiralRepClassifier}. In the second, $G$ does not have SCP by Corollary
\ref{cor:redToCyclic4} and in the third, $G$ does not have SCP by Proposition \ref{pr:KsmChiral}, so this theorem
establishes Conjecture \ref{conj:subgroup} when $n=4$.

The proof of Theorem \ref{th:subgroup} is by an analysis of cases, and takes up Sections \ref{sec:nontoroidal} and \ref{sec:toroidal}.

Conjugacy classes of finite subgroups of $\Og(4)$ have been classified in several places. The reference
\cite{Conway2003} is corrected slightly in \cite{Rastanawi2022}, and we use the tables of both. In the language of
\cite{Rastanawi2022}, there are four classes of finite subgroups of $\Og(4)$, to wit:
\begin{enumerate}
    \item The \emph{toroidal groups}, which leave invariant two orthogonal $2$-planes $V,W$ in $\R^4$, and therefore leave invariant two circles $S_V,S_W$ in $S^3$ that are \emph{absolutely orthogonal} in the sense of \cite{Rastanawi2022}, and therefore leave invariant the torus of points equidistant from $S_V$ and $S_W$. This class of groups consists of several $1$-, $2$- or $3$-parameter infinite families of groups, and are the most troublesome to deal with.
    \item The \emph{tubical groups}, that leave invariant a Hopf fibration of $S^3$ but not a torus in $S^3$. This class consists of a smaller number of $1$-parameter infinite families, and are more easily dealt with.
    \item The \emph{axial groups} that are congruent to non-toroidal finite subgroups of $\Og(3) \times \Og(1)$. Such
      groups have LCP, by \ref{pr:chiralRepClassifier}, and so Conjecture \ref{conj:subgroup} holds for them without further thought.
    \item The \emph{polyhedral groups}, a family of $25$ sporadic groups that are related to symmetries of regular $4$-dimensional polytopes.
\end{enumerate}

\begin{remark}
  Some of these groups $G$ are \emph{metachiral} in the sense of \cite{Conway2003}. That is, if
  $w \in \Og(4) \sm \SO(4)$, then $w^{-1}Gw \neq G$. A metachiral group cannot have LCP, since $w^{-1}gw \neq g$ for
  some $g \in G$.
\end{remark}

\section{Non-toroidal groups}

\label{sec:nontoroidal}

\subsection{Tubical groups}

The left-tubical groups, in the terminology of \cite{Rastanawi2022}, are the $\SO(4)$-conjugacy classes of groups that
in \cite{Conway2003} take the form $\pm \frac{1}{f} [X \times Y]$, where $X \in \{T, O, I\}$ and $Y$ is cyclic or
dihedral. These are listed, along with generators for specific representatives, in \cite[Table 2]{Rastanawi2022} or
\cite[Table 4.1]{Conway2003}. The right-tubical groups are obtained from the left-tubical groups by conjugating by an
element of $\Og(4)\sm \SO(4)$.

\begin{proposition}
    Theorem \ref{th:subgroup} holds for the tubical groups.
\end{proposition}
\begin{proof}
    Each tubical group is conjugate in $\Og(4)$ to one that contains either $[i,1]$ or $[\w, 1]$ or both, by considering the generators in \cite[Table 4.1]{Conway2003}. Since neither $i$ nor $\w$ is conjugate to $1$ in $\HH^\times$, Proposition \ref{pr:quaternionCondition} applies.
\end{proof}

\subsection{Polyhedral and axial groups}

The polyhedral and axial groups, in the terminology of \cite{Rastanawi2022}, are the conjugacy classes of subgroups of $\Og(4)$ that in \cite{Conway2003} are named
\[ \pm \frac{1}{f} [ X \times Y] (\cdot 2_?) \quad \text{or } \frac{1}{f}[X \times Y] (\cdot 2_?), \qquad X, Y \in \{T,O,I, \bar T, \bar O, \bar I\}. \]
These are listed among the groups in \cite[Tables 4.1, 4.2 and 4.3]{Conway2003}.
The axial groups are those whose Coxeter name contains no more than two integers different from $2$. None of the ``metachiral'' classes of \cite[Table 4.1]{Conway2003} (for which Coxeter names are not given) is axial, by comparison with \cite[Table 16]{Rastanawi2022}.

\begin{proposition} \label{pr:thForAxialPolyhedral}
    Theorem \ref{th:subgroup} holds for all axial and polyhedral groups.
\end{proposition}
\begin{proof}
Theorem \ref{th:subgroup} holds for axial groups, which have an invariant $1$-dimensional subspace of
$\R^4$ and therefore have LCP. The polyhedral groups are the remaining cases.

Generators for specific representatives of polyhedral groups are given in \cite[Tables 4.1, 4.2 and 4.3]{Conway2003}, although beware that the element $i'_I$ is missing a definition and should be taken to be
\[ i'_I = \frac{1}{2}\Big(- \frac{\sqrt{5}-1}{2}i - \frac{\sqrt{5}+1}{2}j + k \Big) \]
according to \cite[Appendix A]{Rastanawi2022}. In most cases, Proposition \ref{pr:quaternionCondition} applies immediately since $[i,1]$ or $[\w, 1]$ is listed among the generators of the polyhedral group or its intersection with $\SO(4)$ (its ``chiral part"). The exceptional cases are $\pm \frac{1}{60}[I \times \bar I]$, $+\frac{1}{60}[ I \times \bar I]$, and their extensions $\pm \frac{1}{60}[I \times \bar I]\cdot 2$ and $+\frac{1}{60}[I \times \bar I] \cdot 2_1$ or $2_3$. For these groups, the listed generators in \cite{Conway2003} or \cite{Rastanawi2022} do not contain elements to which Proposition \ref{pr:quaternionCondition} applies immediately.

For $\pm \frac{1}{60}[I \times \bar I]$, we take the listed generators $[\w, \w]$ and $[i_I, -i'_I]$ and multiply them to give
\[ [z,w] =  \left[ \frac{1}{4}(-\sqrt{5}-1) - \frac{1}{4}(\sqrt{5}-1)j - \frac{1}{2}k, -\frac{1}{4}(\sqrt{5}-1) -\frac{1}{4}(\sqrt{5}+1)i+\frac{1}{2}k \right]. \]
In this presentation,
\[ z+\bar z = \frac{1}{2}(-\sqrt{5}-1) \neq \frac{1}{2}(-\sqrt{5}+1) = w + \bar w\]
so that $z$ and $w$ are not conjugate in $\HH^\times$ and Proposition \ref{pr:quaternionCondition} applies.

Similarly, for $+\frac{1}{60}[I \times \bar I]$, we take the listed generators $[\w, \w]$ and $[i_I, i'_I]$ and multiply to give
\[ [z,w] = \left[ \frac{1}{4}(-\sqrt{5}-1) - \frac{1}{4}(\sqrt{5}-1)j -\frac{1}{2}k, 
\frac{1}{4}(\sqrt{5}-1)+\frac{1}{4}(\sqrt{5}+1)i-\frac{1}{2}k \right]. \]
Again $z$ and $w$ have different traces, so that Proposition \ref{pr:quaternionCondition} applies.

Theorem \ref{th:subgroup} holds for the groups $\pm{1}{60}[I \times \bar I]\cdot 2$ and $+\frac{1}{60}[I \times \bar I]\cdot 2_{?}$ since they admit the groups just considered as subgroups.
\end{proof}

\section{Toroidal Groups}
\label{sec:toroidal}

The toroidal groups, as defined in \cite{Rastanawi2022}, constitute a large collection of parameterized families of finite subgroups of $\Og(4)$. They are the groups that \cite{Conway2003} derive from $[X \times Y]$ where $X$ and $Y$ are both cyclic or dihedral.

Let $T \subset S^3$ denote the torus consisting of points that are equidistant in the euclidean distance on $\RR^4$ from
the unit circle $\{(x,y,0,0) \mid x^2 + y^2 = 1\}$ and $\{(0,0,z,w) \mid z^2 + w^2 \}$. ``Toroidal'' groups are so-named
because they are conjugate, in $\Og(4)$, to groups that leave $T$ invariant. Fix a basepoint
$t_0 = (2^{-1/2}, 0, 2^{-1/2}, 0) \in T$. Every toroidal group $G$ has a \emph{directional} subgroup $G_{t_0}$
consisting of the stabilizer of $t_0$. By lifting to the universal cover $(\R^2, \vec 0) \to (T , t_0)$, we can identify
$G_{t_0}$ with a subgroup of linear isometries of $\R^2$ that leave the standard lattice $\Z^2$ invariant. That is,
$G_{t_0}$ is a subgroup of $D_8 = \langle \rho, \sigma \mid \rho^4 = \sigma^2 = (\rho \sigma)^2 = e \rangle$. This group
has $8$ subgroups, up to conjugacy, which we then use to classify the toroidal groups further, as in Table
\ref{tab:classToroidal}.
\begin{table}
  \centering
  \begin{tabular}{|c|c||c|c|} \hline
    Directional Group & Name in \cite{Rastanawi2022} & Directional Group & Name in \cite{Rastanawi2022} \\ \hline
    $\{e\}$ & Translation & $\langle \rho\sigma \rangle$ & Swap\\
    $\langle \sigma \rangle$ & Reflection & $\langle \rho \sigma, \rho^2 \rangle $ & Full Swap \\
    $\langle \rho^2 \rangle$ & Flip & $\langle \rho \rangle$ & Swapturn \\
    $\langle \rho^2, \sigma \rangle $& Full Reflection & $D_8$ & Full Torus \\ \hline
  \end{tabular}
  \caption{A classification of toroidal groups by the directional part. In the names, the word ``Torus'' has been omitted.}
  \label{tab:classToroidal}
\end{table}

  Up to conjugacy in $\Og(4)$, a torus translation group may be taken to be generated by two translations
  \[  \left[\exp\left(-\frac{2\pi i}{m}\right),1\right],\left[\exp\left(-\frac{(m+2s)\pi i}{mn}\right),\exp\left(\frac{\pi i}{n}\right)\right],\]
  where $m,n\geq 1$ and $-\frac{m}{2}\leq s\leq -\frac{m}{2}+\frac{n}{2}$, according to \cite[Proposition 7.5]{Rastanawi2022}.

\begin{proposition}\label{pr:translation:group}
  Theorem  \ref{th:subgroup} holds for torus translation groups. 
\end{proposition}
\begin{proof}
  If $m=1$, then the group is cyclic and the result holds by \cite{Boyle2023}\benw{fix up}\lalog{What needs to be fixed?}.
  
  If $m=2$, then the group is generated by $[-1,1]$ and $\left[\exp\left(-\frac{(1+s)\pi i}{n}\right),\exp\left(\frac{\pi i}{n}\right)\right]$. This is again subject to Proposition \ref{pr:quaternionCondition}, provided $-1-s \not \equiv 1 \pmod n$ and, if $n$ is even, $-1-s \not \equiv  n/2+1 \pmod n$.

  Given that $ -1 \le s \le \frac{n}{2} - 1$, the condition $-1-s \equiv 1 \pmod n$ implies that $n \in \{1,2\}$ and  $s = n-2$. These values give us the groups generated by
  \[ \{[-1, 1], [1, -1] \} \quad \text{ and } \quad \{[ -1, 1], [-i,i] \}. \]
  The first is cyclic, so has been handled. The second has $1\R \subset \HH= \R^4$ as an invariant subspace.

  Given that $ -1 \le s \le \frac{n}{2} - 1$, the conditions that $n$ is even and $-1-s \equiv n/2 + 1 \pmod n$ imply that $n=2$ and $s=-1$. These values give us the group generated by $\{[-1,1], [1,i]\}$, to which Proposition \ref{pr:quaternionCondition} applies.  

    If $m \ge 3$, then the first generator is a double rotation and we can use Proposition \ref{pr:quaternionCondition}.%
  \end{proof}

  Up to conjugacy in $\Og(4)$, a torus flip group may be taken to be generated by two translations
  \[  \left[\exp\left(-\frac{2\pi i}{m}\right),1\right],\left[\exp\left(-\frac{(m+2s)\pi i}{mn}\right),\exp\left(\frac{\pi i}{n}\right)\right],  [j,j],\]
  where $m,n\geq 1$ and $-\frac{m}{2}\leq s\leq -\frac{m}{2}+\frac{n}{2}$, according to \cite[\S 7.6]{Rastanawi2022}.

\begin{proposition} \label{pr:flip}
   Theorem \ref{th:subgroup} holds for torus flip groups.
\end{proposition}
\begin{proof}
  Observe that the group in question always contains the torus translation group for the same values of the parameters $(m,n,s)$.
  
  If $m=1$, then the group is generated by
  \[ \left[1,1\right],\left[\exp\left(-\frac{(2s-1)\pi i}{n}\right),\exp\left(\frac{\pi i}{n}\right) \right], [j,j],\]
 so it is dihedral, and has therefore been handled by \cite{Boyle2023}\benw{fix}\lalog{What needs to be fixed?}.

  If $m=2$, then for most values of $(n,s)$, the group contains a subgroup without LCP, by the proof of Proposition \ref{pr:translation:group}. The exceptional cases are listed:
  \begin{enumerate}
  \item $n=1, s=-1$; generators $[-1,1], [1,-1],[j,j]$;
  \item $n=2, s=0$; generators $[-1,1], [-i,i],[j,j]$.
  \end{enumerate}
  In each case, $1\R \subset \HH = \R^4$ is an invariant subspace.
  
  If $m=3$, then the group contains $[\exp(-2\pi/m), 1]$, so that Proposition \ref{pr:quaternionCondition} applies.
\end{proof}

\begin{proposition} \label{pr:torusSwapGroups}
    The torus swap groups of \cite[Table 6]{Rastanawi2022} do not have SCP. In particular, Theorem \ref{th:subgroup} holds for them.
\end{proposition}
These groups are the toroidal groups appearing in the last five rows of \cite[Table 4.1]{Conway2003}.
\begin{proof}
    The proof consists of an inspection of the generators given either in \cite[Table 4.1]{Conway2003} or in \cite[Table 7]{Rastanawi2022}. In all cases, a generator appears to which Proposition \ref{pr:quaternionCondition} applies.
\end{proof}

\begin{proposition} \label{pr:fullTorusSwapGroups}
    The full torus swap groups and the full torus groups of \cite[Table 6]{Rastanawi2022} do not have SCP. In particular, Theorem \ref{th:subgroup} holds for them.
\end{proposition}
These groups are the toroidal groups that depend only on two parameters, $m$ and $n$ in \cite[Table 4.2 and Table 4.3]{Conway2003}.
\begin{proof}
    The proof for the full torus swap groups consists of an inspection of the generators given either in \cite[Table 4.2]{Conway2003} or in \cite[Table 8]{Rastanawi2022}. In all cases, a generator appears to which Proposition \ref{pr:quaternionCondition} applies.
    
    The case of a full torus group $G$ can be deduced that of from $G \cap \SO(4)$, which is a full torus swap group, or it can be seen directly from \cite[Table 8]{Rastanawi2022}.
\end{proof}

What remains are the groups appearing in the last four lines of \cite[Table 4.2]{Conway2003} and in the last four lines of \cite[Table 4.3]{Conway2003}. For some errata concerning these lines, see \cite[Appendix G]{Rastanawi2022}.

\begin{lemma}\label{lm:enem}
    If $m, n$ are positive integers, at least one of which is greater than $2$, then
    \[ [e^{\frac{\pi i}{n}},e^{\frac{\pi i}{n}}][e^{\frac{\pi i}{m}},e^{\frac{-\pi i}{m}}] = \left[e^{\pi i \left(\frac{1}{n} + \frac{1}{m}\right)},e^{\pi i\left(\frac{1}{n} - \frac{1}{m}\right)}\right]\]
    has no real eigenspaces.
\end{lemma}
The proof is an immediate consequence of Proposition \ref{pr:quaternionCondition}.

For the next results we will use the notation of section 7 of \cite[Table 6]{Rastanawi2022}

\begin{proposition}
  For all swapturn groups, Theorem \ref{th:subgroup} holds.
\end{proposition}
\begin{proof}
  These groups are those that act on the torus $T$ by rotations and translations. In the classification of
  \cite{Rastanawi2022}, they form a single $2$-parameter family of groups, parameterized by integers $a,b$ satisfying
  \[ a \ge b \ge 0, \qquad a \ge 2.\]
  The case of $(a,b) = (2,0)$ is degenerate, and yields a torus reflection group. We assume either $a \ge 3$ or $b
  \ge 1$.

  According to  \cite[\S~7.8]{Rastanawi2022}, the generators of a torus swapturn group may be taken to be
  \[ \left[\exp\left(\frac{-(a+b)\pi i}{a^2 + b^2}\right), \exp\left(\frac{(a-b)\pi i}{a^2 + b^2}\right)\right], \,  \left[\exp\left(\frac{(a-b)\pi i}{a^2 +
      b^2}\right), \exp\left(\frac{(a+b)\pi i}{a^2 + b^2}\right)\right], \, \ast[-j,1]. \]
  Proposition \ref{pr:quaternionCondition} and
  Remark \ref{rem:sign} imply that the first generator has no invariant $1$-dimensional subspace unless the real parts of
  $\pm \exp\Big(\frac{-(a+b)\pi i}{a^2 + b^2}\Big)$ and of $\exp\Big(\frac{(a-b)\pi i}{a^2 + b^2}\Big)$ are equal, which reduces by
  elementary trigonometry to the condition:
  \[ \pm (a+b) \equiv a-b \pmod{a^2 + b^2}. \]
  If $b \ge 1$, then since $a \ge 2$, it must be that $0< 2b \le 2a \le a^2+b^2$, so that neither of these congruences
  holds. Therefore there is an element in the group without a real eigenspace

  In the remaining case when $b=0$, we consider instead the product of the first two generators, which is
  \[  \left[1 , \exp\Big(\frac{2\pi i}{a}\Big)\right]. \] 
Since $a \ge 3$ when $b=0$, Proposition \ref{pr:quaternionCondition} and
Remark \ref{rem:sign} apply again to tell us this element has no real eigenspace.
\end{proof}

\begin{proposition} \label{pr:torus:reflection:full}
For all full reflection groups, Theorem \ref{th:subgroup} holds.
\end{proposition}
\begin{proof}
  These groups fall into four subfamilies, all of which are parameterized by pairs $(m,n)$ where $1 \le n \le m$ and $m \neq 1$. In each case, the group contains $[e^{\frac{\pi i}{n}},e^{\frac{\pi i}{n}}]$, $[e^{\frac{\pi i}{m}},e^{\frac{-\pi i}{m}}]$, and therefore Lemma \ref{lm:enem} applies unless $m=2$ and $n\in \{1,2\}$.

  We give generators for the 8 remaining cases.
  \begin{enumerate}
  \item \label{p2mm1} $n=1$, subtype $\mathbf{p^{2mm}}$: $[i,i]$, $*[i,i]$, $*[k,k]$;
  \item \label{p2mm2} $n=2$, subtype $\mathbf{p^{2mm}}$:$[i,i]$, $[i,-i]$, $*[i,i]$, $*[k,k]$;
  \item \label{p2mg1} $n=1$, subtype $\mathbf{p^{2mg}}$: $[i,i]$, $*[i,i][i,-i]$, $*[k,k][i,-i]$;
  \item \label{p2mg2} $n=2$, subtype $\mathbf{p^{2mg}}$:$[i,i]$, $[i,-i]$, $*[i,i][\zeta,\bar \zeta]$, $*[k,k][\zeta, \bar \zeta]$;
    ($\zeta \in \RR \oplus i \RR$ is a primitive $8$-th root of $1$).
 \item \label{p2gg1} $n=1$, subtype $\mathbf{p^{2gg}}$: $[i,i]$, $*[i,i][\zeta^3, \zeta]$, $*[k,k][\zeta^3, \zeta]$;
 \item \label{Kgroup} $n=2$, subtype $\mathbf{p^{2gg}}$:$[i,i]$, $[i,-i]$, $*[i,i][i,1]$, $*[k,k][i,1]$;
 \item $n=1$, subtype $\mathbf{c^{2mm}}$: $[i,i]$, $[\zeta^3, \zeta]$, $*[i,i]$, $*[k,k]$;
 \item $n=2$, subtype $\mathbf{c^{2mm}}$:$[i,i]$, $[i,-i]$, $[i,1]$, $*[i,i]$, $*[k,k]$;
 \end{enumerate}

If $a, b\in \RR \oplus i \RR$, then $[a, b]\cdot j = \bar a j b = j ab$. By direct calculation
$\ast [i,i] \cdot j = -i(-j)i= j$ and similarly $\ast [k,k] \cdot j = j$. Using these calculations, we see that
immediately  that $j \RR$ is an invariant  $1$-dimensional subspace for each group appearing above, except for case
\ref{Kgroup}. This group is the group $K$. Therefore Theorem \ref{th:subgroup} holds for these cases.
\end{proof}

\subsection{The group \texorpdfstring{$K$}{K}}

The group $K$ is generated by $k_1=*[i,i][i,1]$ and $k_2=*[k,k][i,1]$, both of which are of order $4$. Consider two planes $V_1, V_2 \subset \HH=\RR^4$, where $V_1$ is spanned by $\{1, i\}$ and $V_2$ is spanned by $\{j,k\}$. Endow both $V_1$ and $V_2$ with orientations based on the given ordering of the basis vectors.

The element $k_1$ acts on $\R^4$ by interchanging $1$ and $-i$ and by rotating $V_2$ by $1/4$ turn. The element $k_2$ acts by interchanging $j$ and $-k$ and by rotating $V_1$ by $1/4$ turn.

We make some observations about $K$: every element of $K$ leaves $V_1$ and $V_2$ invariant. The element $k_1^2$ acts as $-1$ on $V_1$ and as the identity on $V_2$, and similarly $k_2^2$ acts as the identity on $V_2$ and as $-1$ on $V_2$. Notably, $k_1^2, k_2^2$ are central in $K$ and $k_1^2k_2^2=-1$. By direct calculation, $k_1k_2 = -k_2k_1$ and so every element of $K$ can be written uniquely in the form $k_1^{e_1} k_2^{e_2}$ where $e_1, e_2 \in \{0, 1, 2, 3\}$.

\begin{proposition} \label{pr:KAllRealEig}
    Every element of $K$ has a real eigenvalue.
\end{proposition}
\begin{proof}
    The element $k_1$ has $1-i\in \HH$ as an eigenvector, as has $k_2^2$, with which it commutes. Therefore all elements $k_1^{e_1} k_2^{2e_2}$ have a real eigenvalue. A similar argument applies to elements of the form $k_1^{2e_1}k_2^{e_2}$, which have $j-k$ as an eigenvector.

    Only the four elements $k_1k_2$, $k_1k_2^3$, $k_1^3k_2$ and $k_1^3k_2^3$ remain to be checked. All these have $1 \in \HH$ as an eigenvector.
\end{proof}

\begin{proposition} \label{pr:KNoCommon}
    There is no $1$-dimensional subspace of $\RR^4=\HH$ that is $K$-invariant.
\end{proposition}
\begin{proof}
    The only $1$-dimensional $k_1$-invariant subspaces of $\HH$ are $(1+i)\RR$ and $(1-i)\RR$, whereas the only $1$-dimensional $k_2$-invariant subspaces of $\HH$ are $(j+k)\RR$ and $(j-k)\RR$.
\end{proof}

\begin{proposition} \label{pr:KsmChiral}
  The group $K$ does not have SCP.
\end{proposition}
\begin{proof}
  We first claim that there are precisely $2$ round circles in $S^3$ that are invariant under the action of $K$. Recall
  that a round circle is the nonempty intersection of $S^3$ with a non-tangent plane in $\RR^4$. The centre of any round
  circle is a point in $\RR^4$ and if the circle is $K$-invariant, then the centre is $K$-fixed.

    The only fixed point of the $K$-action on $\RR^4$ is the origin, so we see that all $K$-invariant round circles in
    $S^3$ must be the intersection of $S^3$ with a $2$-dimensional linear subspace of $\RR^4$. Invariance of the circle
    under $K$ implies that the subspace must also be invariant under $K$. There are exactly two $K$-invariant
    $2$-dimensional subspaces of $\RR^4$, being $V_1$ and $V_2$---these being the only $2$-dimensional subspaces that
    are invariant under $k_1$.

    Therefore the claim holds, and the $2$ circles in question are $S_1 := S^3 \cap V_1$ and $S_2 := S^3 \cap V_2$.

    Let us orient $S_1$ in such a way that $k_2$ acts as a negative $1/4$-turn of $S_1$, and let us orient $S_2$ in such
    a way that $k_1$ acts as a negative $1/4$-turn. Let us orient $S^3$ in such a way that $\link(S_1, S_2) = 1$.

    Now, suppose for the sake of contradiction that $w \in \Diff(S^3) \sm \Diff^+(S^3)$ is such that $wkw^{-1} = k$ for
    all $k \in K$. Consider the oriented circles $wS_1$ and $wS_2$, which are $K$-invariant but not necessarily
    round. Since $w$ is orientation-reversing, $\link(wS_1, wS_2) = -1$.

    By the main result of \cite{Freedman1995}, $wS_1$ is smoothly isotopic, through $K$-invariant circles, to a
    $K$-invariant round circle $S'$. Since $k_2$ acts as a $-1/4$-turn on $wS_1$, \cite[Prop.~4.3]{Boyle2023} tells us
    that $k_2$ acts as a $-1/4$-turn on all the circles in the isotopy, and in particular on $S'$. Therefore $S' = S_1$,
    as oriented circles.
    
    Since $k_2^2$ acts nontrivially on every circle appearing in the isotopy, but trivially on $wS_2$, it follows that
    the isotopy occurs in $S^3 \sm wS_2$, and so
    \[ -1= \link(wS_1, wS_2) = \link(S_1, wS_2). \]

    We now repeat this isotopy argument, with $wS_2$ in the place of $wS_1$, with $k_1$ in the place of $k_2$ and with $S_1$ in
    the place of $wS_2$. We deduce that there is an isotopy from $wS_2$ to $S_2$ in the complement of
    $S_1$. Therefore
    \[ -1 = \link(S_1, wS_2) = \link(S_1, S_2) = 1, \] which is a contradiction. This proves that $K$ does not have SCP.
  \end{proof}

\printbibliography
\end{document}